\documentclass[12pt]{article}
\usepackage{paralist,amsmath,amsthm,amssymb,mathtools,url,graphicx,pdfpages,tikz,pdfpages,rotating}
\usepackage{cancel}
\usepackage[affil-it]{authblk}
\usepackage[left=1in,right=1in,top=1.2in, bottom=1in]{geometry}
\usepackage{mathdots}
\usepackage{comment}
\usepackage{bm} 

\newtheorem{theorem}{Theorem}[section]
\newtheorem{lemma}[theorem]{Lemma}
\newtheorem{corollary}[theorem]{Corollary}

\theoremstyle{definition}

\newtheorem{example}[theorem]{Example}
\newtheorem{obs}[theorem]{Observation}

\newtheorem{prop}[theorem]{Proposition}

\numberwithin{equation}{section}

\begin{document}
\title{The Inverse of the Incidence Matrix of a Unicyclic Graph}
\author{Ryan Hessert} 
\author{Sudipta Mallik\thanks{Corresponding author}} 
\affil{\small Department of Mathematics and Statistics, Northern Arizona University, 801 S. Osborne Dr.\\ PO Box: 5717, Flagstaff, AZ 86011, USA 

rph53@nau.edu, sudipta.mallik@nau.edu}

\maketitle
\begin{abstract}
The vertex-edge incidence matrix of a (connected) unicyclic graph $G$ is a square matrix which is invertible if and only if the cycle of $G$ is an odd cycle.  A combinatorial formula of the inverse of the incidence matrix of an odd unicyclic graph was known. A combinatorial formula of the Moore-Penrose inverse of the incidence matrix of an even unicyclic graph is presented solving an open problem.
\end{abstract}

\section{Introduction}
The {\it Moore–Penrose inverse} of an $m\times n$ real matrix $A$, denoted by $A^+$, is the $n\times m$ real matrix that satisfies the following equations \cite{BG}:
$$AA^+A=A, A^+AA^+=A^+, (AA^+)^T=AA^+, (A^+A)^T=A^+A.$$
When $A$ is invertible, $A^+=A^{-1}$.

Let $G$ be a simple graph on $n$ vertices $1,2,\ldots,n$ with $m$ edges $e_1,e_2,\ldots,e_m$. The vertex-edge {\it incidence matrix} of $G$, denoted by $M$, is the $n\times m$ matrix whose $(i,j)$-entry is $1$ if vertex $i$ is incident with edge $e_j$ and $0$ otherwise. When $G$ is connected, the distance between its vertices $i$ and $j$, denoted by $d(i,j)$,  is the minimum number of edges in a path between $i$ and $j$. 
\begin{obs}\label{Obs on M}
Let $G$ be a connected graph on $n$ vertices $1,2,\ldots,n$ with $m$ edges and the incidence matrix $M$. If $G$ has no odd cycles (i.e., $G$ is bipartite), then $$[(-1)^{d(i,j)}] M=O_{n,m}.$$
\end{obs}
\begin{proof}
The $(i,j)$-entry of $[(-1)^{d(i,j)}]M$ is
$(-1)^{d(i,r)}+(-1)^{d(i,s)}$
where edge $e_j=\{r,s\}$. If $G$ has no odd cycles, then $d(i,r)=d(i,s)\pm 1$ which implies $(-1)^{d(i,r)}+(-1)^{d(i,s)}=0$.
\end{proof}

In 1965,  Ijira first studied the Moore-Penrose inverse of the oriented incidence matrix of a graph in \cite{I}. Bapat did the same for the Laplacian and the edge-Laplacian of trees \cite{B}. Further research studied the same topic for different graphs such as distance regular graphs \cite{AB,ABE}. Meanwhile the signless Laplacian of graphs started being an active area of research \cite{CRC, HM}.  Hessert and Mallik studied the Moore-Penrose inverses of the incidence matrix and  the signless Laplacian of a tree and an odd unicyclic graph in \cite{Hessert1}. In particular, they provided the following theorem about the Moore-Penrose inverse of the incidence matrix of a connected graph:

\begin{theorem}\cite[Theorem 2.15]{Hessert1}\label{M+}
Let $G$ be a connected graph on $n$ vertices $1,2,\ldots,n$ with the incidence matrix $M$. 
\begin{enumerate}
    \item[(a)] If $G$ has an odd cycle, then $MM^+=I_n$. 
    
    \item[(b)] If $G$ has no odd cycles (i.e., $G$ is bipartite), then $$MM^+=I_n-\frac{1}{n}[(-1)^{d(i,j)}].$$
\end{enumerate}
\end{theorem}

A {\it unicyclic graph} on $n$ vertices is a simple connected graph that has a unique cycle as a subgraph. A unicyclic graph on $n$ vertices has $n$ edges. From the preceding theorem, we have the following observation.

\begin{obs}
Let $G$ be a unicyclic graph with the incidence matrix $M$. Then $M$ is invertible if and only if $G$ is an odd unicyclic graph.
\end{obs}

A combinatorial formula of the inverse of the incidence matrix of an odd unicyclic graph is given by the following theorem.

\begin{theorem}\cite[Theorem 3.1]{Hessert1}
Let $G$ be an odd unicyclic graph on $n$ vertices $1,2,\ldots,n$ and edges $e_1,e_2,\ldots,e_n$ with the cycle $C$ and the incidence matrix $M$. Then $M$ is invertible and its inverse $M^{-1}=[a_{i,j}]$ is given by
\begin{equation}\label{unicyclic M inverse}\nonumber
a_{i,j}= \begin{cases} 
\frac{(-1)^{d(e_i,j)}}{2}  & \text{ if } e_i \in C\\
0 & \text{ if } e_i \not \in C \text{ and } j \in G\setminus e_i [C]\\
(-1)^{d(e_i,j)}  & \text{ if } e_i \not \in C \text{ and } j \not \in G\setminus e_i [C].
\end{cases} 
\end{equation}
\end{theorem}

The notation $G\setminus e_i [C]$ and $G\setminus e_i (C)$ are described in Section 2. The open problem of finding a combinatorial formula of the Moore-Penrose inverse of the incidence matrix of an even unicyclic graph was posed in \cite[Open Problem 1(a)]{Hessert1}. We solve this open problem in this article.

\section{Main Results}
For a graph $G$, $|G|$ denotes the number of vertices of $G$, i.e., $|G|=|V(G)|$ where $V(G)$ is the vertex set of $G$. Let $G$ be an even unicyclic graph on $n$ vertices $1,2,\ldots,n$ with $n$ edges $e_1,e_2,\ldots,e_n$. Suppose $C$ is the even cycle in $G$. We use the following notation from \cite{Hessert1}: For an edge $e_i$ not in $C$, $G\setminus e_i$ has two connected components. The connected component of $G\setminus e_i$ that contains $C$ is denoted by $G\setminus e_i[C]$. Similarly the connected component of $G\setminus e_i$ that does not contain $C$ is denoted by $G\setminus e_i(C)$. When $e_i$ is on $C$, $G\setminus e_i[C]$ and $G\setminus e_i(C)$ are defined to be $G\setminus e_i$ and the empty graph, respectively.  The unique shortest path between a vertex $i$ and $C$ is denoted by $P_{i-C}$. The shortest distance between vertex $j$ and a vertex incident with edge $e_i$ is denoted by $d(e_i,j)$. When $e_i=\{r_i,s_i\}\in C$, $d_{G\setminus e_i}(r_i,j)$ denotes the distance between vertices $r_i$ and $j$ in the tree $G\setminus e_i$. \\

Now we introduce an $n\times n$ matrix $H$ whose rows and columns are indexed by the edges and vertices of the even unicyclic graph $G$, respectively, and $H=[h_{i,j}]$ is defined as follows:
\begin{equation}\label{m^+ formula}
h_{i,j}=\frac{1}{n|C|} \begin{cases} 
(-1)^{d(e_i,j)}|C||G\setminus e_i[C]|  & \text{ if } e_i \notin C \text{ and } j \in G\setminus e_i (C)\\
(-1)^{d(e_i,j)} |C||G\setminus e_i(C)| & \text{ if } e_i \not \in C \text{ and } j \in G\setminus e_i [C]\\
(-1)^{d_{G\setminus e_i}(r_i,j)} \left(-nd_{G\setminus e_i}(r_i,j^*) +\sum\limits_{t\in C} n_t d_{G\setminus e_i}(r_i,t)  \right)& \text{ if } e_i=\{r_i,s_i\} \in C \text{ and } j \in G, 
\end{cases}
\end{equation}
where $j^*$ is the vertex on the cycle $C$ closest to vertex $j$ and $n_t$ is the number of vertices in the tree branch of $G$ starting with vertex $t\in C$.

Note that with the preceding definition of $n_t$, we have
\[\sum_{t\in C} n_t=n.\]

For example, for the graph on Figure 1, $n_4+n_5+n_7+n_9=3+4+1+1=9$.

\begin{example}
\begin{figure}
\centering
\begin{tikzpicture}[shorten > = 1pt, auto, node distance = .5cm ]
\tikzset{vertex/.style = {shape = circle, draw, minimum size = 1em}}
\tikzset{edge/.style = {-}}
\node[vertex] (1) at (4.5,-1){$1$};
\node[vertex] (2) at (4.5,1){$3$};
\node[vertex] (3) at (3,0){$2$};
\node[vertex] (4) at (1,0){$5$};
\node[vertex] (5) at (-1,1){$7$};
\node[vertex] (6) at (-1,-1){$9$};
\node[vertex] (7) at (-3,0){$4$};
\node[vertex] (8) at (-5,1){$8$};
\node[vertex] (9) at (-5,-1){$6$};
\draw[edge] (1) edge node[below]{$e_3$} (3);
\draw[edge] (2) edge node[above]{$e_5$} (3);
\draw[edge] (3) edge node[above]{$e_4$} (4);
\draw[edge] (4) edge node[above]{$e_6$} (5);
\draw[edge] (5) edge node[above]{$e_9$} (7);
\draw[edge] (7) edge node[below]{$e_7$} (6);
\draw[edge] (8) edge node[above]{$e_1$} (7);
\draw[edge] (9) edge node[below]{$e_2$} (7);
\draw[edge] (6) edge node[below]{$e_8$} (4);
\end{tikzpicture}
\caption{An even unicyclic graph}
\label{fig:even unicyclic}
\end{figure}
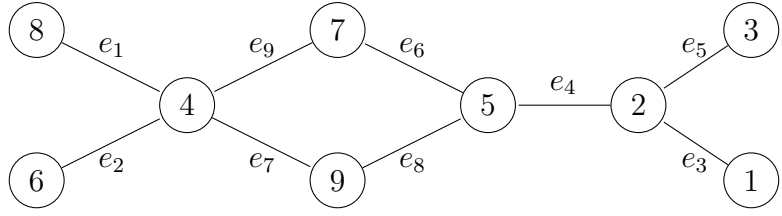

For the even unicyclic graph in Figure 1,
\[M=\left[\begin{array}{rrrrrrrrr}
0 & 0 & 1 & 0 & 0 & 0 & 0 & 0 & 0 \\
0 & 0 & 1 & 1 & 1 & 0 & 0 & 0 & 0 \\
0 & 0 & 0 & 0 & 1 & 0 & 0 & 0 & 0 \\
1 & 1 & 0 & 0 & 0 & 0 & 1 & 0 & 1 \\
0 & 0 & 0 & 1 & 0 & 1 & 0 & 1 & 0 \\
0 & 1 & 0 & 0 & 0 & 0 & 0 & 0 & 0 \\
0 & 0 & 0 & 0 & 0 & 1 & 0 & 0 & 1 \\
1 & 0 & 0 & 0 & 0 & 0 & 0 & 0 & 0 \\
0 & 0 & 0 & 0 & 0 & 0 & 1 & 1 & 0
\end{array}\right]\]
and
\[H = \frac{1}{36} \left[\begin{array}{rrrrrrrrr}
4 & -4 & 4 & 4 & 4 & -4 & -4 & 32 & -4 \\
4 & -4 & 4 & 4 & 4 & 32 & -4 & -4 & -4 \\
32 & 4 & -4 & -4 & -4 & 4 & 4 & 4 & 4 \\
-24 & 24 & -24 & 12 & 12 & -12 & -12 & -12 & -12 \\
-4 & 4 & 32 & -4 & -4 & 4 & 4 & 4 & 4 \\
10 & -10 & 10 & -8 & 10 & 8 & 17 & 8 & -1 \\
-6 & 6 & -6 & 12 & -6 & -12 & -3 & -12 & 15 \\
10 & -10 & 10 & -8 & 10 & 8 & -1 & 8 & 17 \\
-6 & 6 & -6 & 12 & -6 & -12 & 15 & -12 & -3
\end{array}\right].\]
\end{example}

\bigskip
When $e_i=\{r_i,s_i\} \in C$ and $j \in G$, $h_{ij}$ is defined in (\ref{m^+ formula}) using $r_i$. The following proposition shows that $h_{ij}$ is independent of the choice of a vertex of $e_i$:

\begin{prop}\label{def-hij}
Let $G$ be an even unicyclic graph on $n$ vertices $1,2,\ldots,n$ with $n$ edges $e_1,e_2,\ldots,e_n$. Suppose $C$ is the cycle of $G$. Let $e_i = \{r_i,s_i\} \in C$ and $j \in G$. Then the following are equal:
\[(-1)^{d_{G\setminus e_i}(r_i,j)} \left(-nd_{G\setminus e_i}(r_i,j^*) +\sum\limits_{t\in C} n_t d_{G\setminus e_i}(r_i,t)  \right)\]
and
\[(-1)^{d_{G\setminus e_i}(s_i,j)} \left(-nd_{G\setminus e_i}(s_i,j^*) +\sum\limits_{t\in C} n_t d_{G\setminus e_i}(s_i,t)  \right).\]
\end{prop}

To prove the preceding result, first we need the following lemma:
\begin{lemma}
\label{parity0}
Let $G$ be an even unicyclic graph on $n$ vertices $1,2,\ldots,n$ with $n$ edges $e_1,e_2,\ldots,e_n$. Suppose $C$ is the cycle of $G$. Let $e_i = \{r_i,s_i\} \in C$, $j \in G$, and $j^*$ be the vertex on $C$ that is closest to vertex $j$.
\begin{enumerate}
    \item[(i)] $d_{G\setminus e_i}(r_i,j^*)$ and $d_{G\setminus e_i}(s_i,j^*)$ have opposite parities.
    \item[(ii)] $d_{G\setminus e_i}(r_i,j)$ and $d_{G\setminus e_i}(s_i,j) $ have opposite parities.
     
\end{enumerate}
\end{lemma}
\begin{proof}

\begin{enumerate}
    \item[(i)]  Since $|C|$ is even, $d_{G \setminus e_i}(r_i,j^*)+d_{G \setminus e_i}(s_i,j^*) = |C|-1$ is odd. Therefore $d_{G \setminus e_i}(r_i,j^*)$ and $d_{G \setminus e_i}(s_i,j^*)$ have opposite parities. 
    
    \item[(ii)] Note that $d_{G\setminus e_i}(r_i,j) = d_{G \setminus e_i}(r_i,j^*) + d(j,j^*)$ and $d_{G\setminus e_i}(s_i,j) = d_{G \setminus e_i}(s_i,j^*) + d(j,j^*)$. Then   
    $d_{G\setminus e_i}(r_i,j)$ and $d_{G\setminus e_i}(s_i,j) $ have opposite parities by (i).
     
    \end{enumerate}
\end{proof}

{\it Proof of Proposition \ref{def-hij}.}
Without loss of generality, suppose that $d_{G\setminus e_i}(r_i,j)$ is even and $d_{G\setminus e_i}(s_i,j)$ is odd by \ref{parity0} (ii). Then it suffices to show that the following are equal:
\[(-1)^{d_{G\setminus e_i}(r_i,j)} \left(-nd_{G\setminus e_i}(r_i,j^*) +\sum\limits_{t\in C} n_t d_{G\setminus e_i}(r_i,t)  \right) 
= -nd_{G\setminus e_i}(r_i,j^*) +\sum\limits_{t\in C} n_t d_{G\setminus e_i}(r_i,t)\]
and 
\[(-1)^{d_{G\setminus e_i}(s_i,j)} \left(-nd_{G\setminus e_i}(s_i,j^*) +\sum\limits_{t\in C} n_t d_{G\setminus e_i}(s_i,t)  \right) 
= nd_{G\setminus e_i}(s_i,j^*) -\sum\limits_{t\in C} n_t d_{G\setminus e_i}(s_i,t).\]

Since $d_{G\setminus e_i}(s_i,j^*) = |C| - 1 -d_{G\setminus e_i}(r_i,j^*)$ and $d_{G\setminus e_i}(s_i,t) = |C| - 1 -d_{G\setminus e_i}(r_i,t)$ for each $t \in C$,  we have \\
\begin{align*}
   &\; nd_{G\setminus e_i}(s_i,j^*) -\sum\limits_{t\in C} n_t d_{G\setminus e_i}(s_i,t)  \\
   =&\; n (|C| - 1 -d_{G\setminus e_i}(r_i,j^*)) -\sum\limits_{t\in C} n_t (|C| - 1 -d_{G\setminus e_i}(r_i,t))\\
   =&\; n(|C|-1) -nd_{G\setminus e_i}(r_i,j^*) -(|C|-1)\sum\limits_{t\in C} n_t 
   +\sum\limits_{t\in C}n_t d_{G\setminus e_i}(r_i,t)\\
   =&\; n(|C| -1) -nd_{G\setminus e_i}(r_i,j^*) -n(|C|-1) + \sum\limits_{t\in C}n_t d_{G\setminus e_i}(r_i,t)\\
   =&-nd_{G\setminus e_i}(r_i,j^*) + \sum\limits_{t\in C}n_t d_{G\setminus e_i}(r_i,t).
\end{align*}
\qed

Now we show that the matrix $H$ defined in (\ref{m^+ formula}) is the Moore-Penrose inverse of the incidence matrix of the corresponding even unicyclic graph. First we need the following results using the following notation: When there are unique shortest paths from vertex $i$ to vertex $j$ and edge $e_j$, they are denoted by $P_{i-j}$ (or $P_{j-i}$) and $P_{e_j-i}$ (or $P_{i-e_j}$), respectively.

\begin{lemma}\label{parity1}
Let $G$ be an even unicyclic graph on $n$ vertices $1,2,\ldots,n$ with the cycle $C$. Let $e_i$ be an edge and  $j$ be a vertex of $G$. 
\begin{enumerate}
    \item[(a)] Let $e_i\notin C$ and $j\in G\setminus e_i (C)$. If $k\in G\setminus e_i (C)$, then $(-1)^{d(e_i,k)+d(k,j)}=(-1)^{d(e_i,j)}$.
    If $k\in G\setminus e_i [C]$, then  $(-1)^{d(e_i,k)+d(k,j)}=-(-1)^{d(e_i,j)}$.
   
    \item[(b)] Let $e_i\notin C$ and $j\in G\setminus e_i [C]$. If $k\in G\setminus e_i (C)$, then $(-1)^{d(e_i,k)+d(k,j)}=-(-1)^{d(e_i,j)}$.
    If $k\in G\setminus e_i [C]$, then $(-1)^{d(e_i,k)+d(k,j)}=(-1)^{d(e_i,j)}$.
    
    \item[(c)] Let $e_i=\{r_i,s_i\}\in C$ and $j\in G$. Then for any vertex $k\in G$,
    $$(-1)^{d_{G\setminus e_i}(r_i,k)+d(k,j)}=(-1)^{d_{G\setminus e_i}(r_i,j)}.$$
    
    \item[(d)] Let $e_i=\{r_i,s_i\}\in C$. Then
    $$\sum\limits_{t\in C} n_t d_{G\setminus e_i}(r_i,t)
    =\sum_{k=1}^n d_{G\setminus e_i}(r_i,k^*).$$
    
\end{enumerate}
\end{lemma}

\begin{proof}
\begin{enumerate}
    \item[(a)] First let $k\in G\setminus e_i [C]$. Then $d(e_i,k)+d(k,j) = d(e_i,j) + 2d(e_i,k) + 1$. So $d(e_i,k)+d(k,j)$ and $d(e_i,j)$ have opposite parities which implies 
    \[(-1)^{d(e_i,k)+d(k,j)}=-(-1)^{d(e_i,j)}.\]
    
    Now let $k\in G\setminus e_i (C)$. It suffices to show that $d(e_i,k)+d(k,j)$ and $d(e_i,j)$ have the same parity. If $k\in P_{e_i-j}$, then $d(e_i,k)+d(k,j) = d(e_i,j)$. Now suppose $k\notin P_{e_i-j}$. Define $k'$ as the vertex on $P_{e_i-k} \cap P_{e_i-j}$ that is closest to $k$. For example, $k'=j$ when $j \in P_{e_i-k}$. In all possible cases for $k'$, $d(e_i,k)+d(k,j) = d(e_i,j) + 2d(k,k')$.
    
    \item[(b)]  The proof is similar to that of (a).
    
    \item[(c)] Let $k$ be a vertex of $G$. First note that $d_{G\setminus e_i}(k,j)$ is either $d(k,j)$ (when there is a shortest path, not necessarily unique, between $k$ and $j$ in $G$ not containing $e_i$) or $|C|+2d(j,j^*)+2d(k,k^*)-d(k,j)$ (when $P_{k-j}$ contains $e_i$). Since $|C|$ is even and $-d(k,j)$ has the same parity as $d(k,j)$, $d_{G\setminus e_i}(k,j)$ and $d(k,j)$ have the same parity. Therefore, it suffices to show that $d_{G\setminus e_i}(r_i,j)$ and $d_{G\setminus e_i}(r_i,k)+d_{G\setminus e_i}(k,j)$ have the same parity. The unique shortest path between vertices $x$ and $y$ in the tree $G \setminus e_i$ is denoted by $P'_{x-y}$ in the following proof.
    
    If $k\in P'_{r_i-j}$, then $d_{G\setminus e_i}(r_i,j)=d_{G\setminus e_i}(r_i,k)+d_{G\setminus e_i}(k,j)$. Now suppose $k\notin P'_{r_i-j}$.
    
    Case 1. $j\in P'_{r_i-k}$\\
    In this case, $d_{G\setminus e_i}(r_i,j)=d_{G\setminus e_i}(r_i,k)-d_{G\setminus e_i}(k,j)$. Since $d_{G\setminus e_i}(k,j)$ and $-d_{G\setminus e_i}(k,j)$ have the same parity, so do $d_{G\setminus e_i}(r_i,j)$ and $d_{G\setminus e_i}(r_i,k)+d_{G\setminus e_i}(k,j)$.
    
    Case 2. $j\notin P'_{r_i-k}$\\
    In this case, $d_{G\setminus e_i}(r_i,k)+d_{G\setminus e_i}(k,j)=d_{G\setminus e_i}(r_i,j)+2d_{G\setminus e_i}(k,\widetilde{k})$ where $\widetilde{k}$ is the vertex on $P'_{r_i-j}$ that is closest to $k$ in $G\setminus e_i$. Thus $d_{G\setminus e_i}(r_i,j)$ and $d_{G\setminus e_i}(r_i,k)+d_{G\setminus e_i}(k,j)$ have the same parity.
    
    \item[(d)] For each vertex $t\in C$, suppose $G_t$ is the tree branch of $G$ starting with $t$. Then the vertices of $G$ are partitioned into vertices of $G_t,\;t\in C$. Then 
    $$\sum_{k=1}^n d_{G\setminus e_i}(r_i,k^*)
    =\sum\limits_{t\in C} \sum\limits_{k\in G_t} d_{G\setminus e_i}(r_i,k^*)
    =\sum\limits_{t\in C} |G_t| d_{G\setminus e_i}(r_i,t)
    =\sum\limits_{t\in C} n_t d_{G\setminus e_i}(r_i,t).$$
    
\end{enumerate}
\end{proof}

\begin{theorem}\label{MH even unicyclic}
Let $G$ be an even unicyclic graph on $n$ vertices $1,2,\ldots,n$ with the incidence matrix $M$. For the matrix $H$ defined in (\ref{m^+ formula}),  we have 
\begin{enumerate}
    \item[(a)] $H[(-1)^{d(i,j)}]=O$. 
    
    \item[(b)] $MH=I_n-\frac{1}{n}[(-1)^{d(i,j)}]$.
\end{enumerate}
\end{theorem}
\begin{proof}
Let $e_1,e_2,\ldots,e_n$ be the edges of $G$.

\begin{enumerate}
    \item[(a)] We prove $H[(-1)^{d(i,j)}]=O$ by the following three cases.\\
    
    Case 1. $e_i\notin C$ and $j\in G\setminus e_i (C)$ \\
    The $(i,j)$-entry of $H[(-1)^{d(i,j)}]$ is given by
    \begin{align*}
    &\hspace{16pt} \frac{1}{n} \left[\sum_{k\in G\setminus e_i (C)}  (-1)^{d(e_i,k)+d(k,j)}|G\setminus e_i[C]|
    + \sum_{k\in G\setminus e_i [C]} (-1)^{d(e_i,k)+d(k,j)} |G\setminus e_i(C)|    \right]\\
    &= \frac{1}{n} \left[\sum_{k\in G\setminus e_i (C)}  (-1)^{d(e_i,j)}|G\setminus e_i[C]|
    + \sum_{k\in G\setminus e_i [C]} -(-1)^{d(e_i,j)} |G\setminus e_i(C)|  \right] (\text{by Lemma } \ref{parity1}(a))\\
    &= \frac{(-1)^{d(e_i,j)}}{n} \left[\sum_{k\in G\setminus e_i (C)} |G\setminus e_i[C]|
    - \sum_{k\in G\setminus e_i [C]} |G\setminus e_i(C)|  \right]\\
    &=  \frac{(-1)^{d(e_i,j)}]}{n} \left[ |G\setminus e_i (C)| |G\setminus e_i[C]| -|G\setminus e_i[C]| |G\setminus e_i (C)|  \right]\\
    &= 0.
    \end{align*}
    
     Case 2. $e_i\notin C$ and $j\in G\setminus e_i [C]$ \\
     We use Lemma \ref{parity1}(b). The proof is similar to that of Case 1.
  
    Case 3. $e_i\in C$ and $j\in G$ \\
    
   The $(i,j)$-entry of $H[(-1)^{d(i,j)}]$ is given by
          \begin{align*}
    &\hspace{16pt} \frac{1}{n|C|} \sum_{k=1}^n (-1)^{d_{G\setminus e_i}(r_i,k)+d(k,j)} \left(-nd_{G\setminus e_i}(r_i,k^*) +\sum\limits_{t\in C} n_t d_{G\setminus e_i}(r_i,t)  \right)  \\ 
    =& \frac{1}{n|C|} \sum_{k=1}^n (-1)^{d_{G\setminus e_i}(r_i,j)} \left(-nd_{G\setminus e_i}(r_i,k^*) +\sum\limits_{t\in C} n_t d_{G\setminus e_i}(r_i,t)  \right) (\text{by Lemma } \ref{parity1}(c))\\
    =& \frac{(-1)^{d_{G\setminus e_i}(r_i,j)}}{n|C|}   \left(-n \sum_{k=1}^n d_{G\setminus e_i}(r_i,k^*) +\sum_{k=1}^n\sum\limits_{t\in C} n_t d_{G\setminus e_i}(r_i,t)  \right)\\
    =& \frac{(-1)^{d_{G\setminus e_i}(r_i,j)}}{n|C|}   \left( 
    -n\sum_{k=1}^n d_{G\setminus e_i}(r_i,k^*)
    +n\sum\limits_{t\in C} n_t d_{G\setminus e_i}(r_i,t)\right)\\
    =& 0.  (\text{by Lemma } \ref{parity1}(d))
    \end{align*}
   
    \item[(b)] The $(i,j)$-entry of $MH$ is given by
$$(MH)_{i,j} 
=\sum_{t=1}^k h_{p_t,j},$$
where $e_{p_1},e_{p_2},\ldots,e_{p_k}$ are all the edges incident with vertex $i$.\\ 

First suppose $i=j$.  

Case 1. $i\notin C$\\
Let $e_{p_1} \in P_{i-C}$. Then \begin{align*}
           \sum_{t=1}^k h_{p_t,i} &= h_{p_1,i} + \sum_{t=2}^k h_{p_t,i}\\
          &=  (-1)^{d(e_{p_1},i)} \frac{|G \setminus e_{p_1} [C]|}{n} + \sum_{t=2}^k (-1)^{d(e_{p_t},i)} \frac{|G \setminus e_{p_t} (C)|}{n}\\ 
          &=  \frac{1}{n}\left( |G \setminus e_{p_1} [C]|+ \sum_{t=2}^k |G \setminus e_{p_t} (C)| \right)\\ 
          &=  \frac{1}{n}\left( |G \setminus e_{p_1} [C]|+ |G \setminus e_{p_1} (C)| - 1 \right)\\
          &=\frac{n-1}{n}.
        \end{align*}

Case 2. $i\in C$\\ 
Let $e_{p_1}$ and $e_{p_2}$ be the edges on $C$ that are incident with vertex $i$. 
Then \begin{align*}
    &\;\;\sum_{t=1}^k h_{p_t,i}\\
    =& h_{p_1,i} + h_{p_2,i} + \sum_{t=3}^k h_{p_t,i}\\
    =& \frac{(-1)^{d_{G \setminus e_{p_1}}(i,i)}}{n |C|} \left(-n d_{G \setminus e_{p_1}}(i,i) + \sum_{t \in C} n_t d_{G \setminus e_{p_1}}(i,t) \right)\\
    &+ \frac{(-1)^{d_{G \setminus e_{p_2}}(i,i)}}{n |C|} \left(-n d_{G \setminus e_{p_2}}(i,i) + \sum_{t \in C} n_t d_{G \setminus e_{p_2}}(i,t) \right)
    + \sum_{t=3}^k \frac{(-1)^{d(e_{p_t},i)} |C| |G \setminus e_{p_t} (C)|}{n |C|}.
    \end{align*}
Since  $d_{G \setminus e_{p_1}}(i,i)=d_{G \setminus e_{p_2}}(i,i)=d(e_{p_t},i)=0$ for $t=3,4,\ldots,k$, the above becomes   
    \begin{align*}
    & \frac{1}{n |C|}  \sum_{t \in C} n_t d_{G \setminus e_{p_1}}(i,t)
    + \frac{1}{n |C|}  \sum_{t \in C} n_t d_{G \setminus e_{p_2}}(i,t)
    + \sum_{t=3}^k \frac{|C| |G \setminus e_{p_t} (C)|}{n |C|}\\
    =& \frac{1}{n |C|} \left(\sum_{t \in C} n_t d_{G \setminus e_{p_1}}(i,t) + \sum_{t \in C} n_t d_{G \setminus e_{p_2}}(i,t) + |C| \sum_{t=3}^k |G \setminus e_{p_t} (C)|\right)\\
    =& \frac{1}{n |C|} \left( \sum_{t \in C, t \neq i} n_t (d_{G \setminus e_{p_1}}(i,t) + d_{G \setminus e_{p_2}}(i,t)) + |C| (n_i - 1)\right)\\
    =& \frac{1}{n |C|} \left(\sum_{t \in C, t \neq i} n_t |C| + |C| (n_i - 1)\right)\\
    =& \frac{1}{n} \left( \sum_{t \in C, t \neq i} n_t +  (n_i - 1)\right)\\
    =& \frac{1}{n} \left( (n - n_i) +  (n_i - 1)\right)\\
    =& \frac{n-1}{n}.
\end{align*}

Now suppose $i\neq j$. Without loss of generality, let $e_{p_1}$ be on  a shortest $i-j$ path.

Case 1. $i\notin C$ and $j\notin C$\\
        
        Subcase (i) $j\notin P_{i-C}$, $i\notin P_{j-C}$\\ 
        Then \begin{align*}
          \sum_{t=1}^k h_{p_t,j} &= h_{p_1,j} + \sum_{t=2}^k h_{p_t,j}\\
          &=  (-1)^{d(e_{p_1},j)} \frac{|G \setminus e_{p_1} (C)|}{n} + \sum_{t=2}^k (-1)^{d(e_{p_t},j)} \frac{|G \setminus e_{p_t} (C)|}{n}.
          \end{align*}
Since $d(e_{p_1},j)=d(i,j)-1$ and $d(e_{p_t},j)=d(i,j)$ for $t=2,3,\ldots,k$, the above becomes          
          \begin{align*}
          &  -(-1)^{d(i,j)} \frac{|G \setminus e_{p_1} (C)|}{n} + \sum_{t=2}^k (-1)^{d(i,j)} \frac{|G \setminus e_{p_t} (C)|}{n}\\
          =&  \frac{(-1)^{d(i,j)}}{n} \left( -|G \setminus e_{p_1} (C)| + \sum_{t=2}^k |G \setminus e_{p_t} (C)| \right)\\
          =& \frac{(-1)^{d(i,j)}}{n} \left( -|G \setminus e_{p_1} (C)| + |G \setminus e_{p_1} (C)| - 1 \right)\\
          =&  \frac{-(-1)^{d(i,j)}}{n}.
        \end{align*}
        
        Subcase (ii) $i\in P_{j-C}$\\ 
        Let $e_{p_2} \in P_{i-C}$.  Then \begin{align*}
           \sum_{t=1}^k h_{p_t,j} &= h_{p_1,j} + h_{p_2,j} + \sum_{t=3}^k h_{p_t,j}\\
           &=  (-1)^{d(e_{p_1},j)} \frac{|G \setminus e_{p_1} [C]|}{n} + (-1)^{d(e_{p_2},j)} \frac{|G \setminus e_{p_2} [C]|}{n} + \sum_{t=3}^k (-1)^{d(e_{p_t},j)} \frac{|G \setminus e_{p_t} (C)|}{n}.
           \end{align*}
           
Since $d(e_{p_1,j})=d(i,j)-1$ and $d(e_{p_t,j})=d(i,j)$ for $t=2,3,\ldots,k$, the above becomes
           \begin{align*}
           &\;\; -(-1)^{d(i,j)} \frac{|G \setminus e_{p_1} [C]|}{n} + (-1)^{d(i,j)} \frac{|G \setminus e_{p_2} [C]|}{n} + \sum_{t=3}^k (-1)^{d(i,j)} \frac{|G \setminus e_{p_t} (C)|}{n} \\
           &= \frac{(-1)^{d(i,j)}}{n}\left( -|G \setminus e_{p_1} [C]| + |G \setminus e_{p_2} [C]| + \sum_{t=3}^k |G \setminus e_{p_t} (C)| \right).
           \end{align*}
          
Since $|G\setminus e_{p_2} (C)| = 1+|G \setminus e_{p_1} (C)| +\sum_{t=3}^k |G \setminus e_{p_t} (C)|$, the above becomes
           \begin{align*}           &\;\;\;\;\frac{(-1)^{d(i,j)}}{n}\left( -|G \setminus e_{p_1} [C]| + |G \setminus e_{p_2} [C]| + \left(|G \setminus e_{p_2} (C)| - |G \setminus e_{p_1} (C)| - 1 \right) \right) \\
          &=\frac{(-1)^{d(i,j)}}{n}\left( \left(|G \setminus e_{p_2} [C]|+|G \setminus e_{p_2} (C)|\right) -\left(|G \setminus e_{p_1} [C]| +  |G \setminus e_{p_1} (C)| \right) - 1 \right) \\
           &= \frac{(-1)^{d(i,j)}}{n} (n-n-1)\\
           &= \frac{-(-1)^{d(i,j)}}{n}.
        \end{align*}
        
        Subcase (iii) $j\in P_{i-C}$.\\
        The proof is similar to that of Subcase (ii). \\
       
Case 2. $i\in C$ and $j\in C$\\
    Let $e_{p_2} \in C$.
\begin{align*}
    &\sum_{t=1}^k h_{p_t,j}\\
    =& h_{p_1,j} + h_{p_2,j} + \sum_{t=3}^k h_{p_t,j}\\
    =&  \frac{(-1)^{d_{G \setminus e_{p_1}}(i, j)}}{n|C|} \left(-n d_{G \setminus e_{p_1}} (i, j) + \sum_{t \in C} n_t d_{G \setminus e_{p_1}}(i, t) \right) \\
    &+ \frac{(-1)^{d_{G \setminus e_{p_2}}(i, j)}}{n|C|} \left(-n d_{G \setminus e_{p_2}} (i, j) + \sum_{t \in C} n_t d_{G \setminus e_{p_2}}(i, t) \right) + \sum_{t=3}^k \frac{(-1)^{d(i,j)}}{n|C|} |C||G \setminus e_t (C)|
    \end{align*}
Since $d_{G \setminus e_{p_1}}(i, j)=|C| - d(i, j)$ which has the same parity as $d(i,j)$, the above becomes   
    \begin{align*}
    & \frac{(-1)^{d(i, j)}}{n|C|}\left[ \left(-n(|C| - d(i, j)) + \sum_{t \in C} n_t d_{G \setminus e_{p_1}}(i, t) \right) \right.\\
    &+ \left(-n d(i, j) + \sum_{t \in C} n_t d_{G \setminus e_{p_2}}(i, t) \right) + \left. \sum_{t=3}^k  |C||G \setminus e_t (C)| \right] \\
    =& \frac{(-1)^{d(i, j)}}{n|C|} \left[ -n|C| + \sum_{t \in C} n_t d_{G \setminus e_{p_1}}(i, t) + \sum_{t \in C} n_t d_{G \setminus e_{p_2}}(i, t) + \sum_{t=3}^k |C||G \setminus e_t (C)|  \right] \\ 
    =& \frac{(-1)^{d(i, j)}}{n|C|} \left[ -n|C| + \sum_{t \in C, t \neq i} n_t (d_{G \setminus e_{p_1}}(i, t) + d_{G \setminus e_{p_2}}(i, t)) + \sum_{t=3}^k |C||G \setminus e_t (C)|  \right] \\
    =& \frac{(-1)^{d(i, j)}}{n|C|} \left[ -n|C| + \sum_{t \in C, t \neq i} n_t |C| + |C|\sum_{t=3}^k |G \setminus e_t (C)|  \right] \\
    =& \frac{(-1)^{d(i, j)}}{n} \left[ -n + \sum_{t \in C, t \neq i} n_t  + \sum_{t=3}^k |G \setminus e_t (C)|  \right] \\
    =& \frac{(-1)^{d(i, j)}}{n} \left[ -n + (n - n_i)  + (n_i - 1)  \right] \\
        =& \frac{-(-1)^{d(i, j)}}{n}.\\
\end{align*}

Case 3. $i\notin C$ and $j\in C$\\
Then 
          \[\sum_{t=1}^k h_{p_t,j} 
          = h_{p_1,j} + \sum_{t=2}^k h_{p_t,j}
          =  (-1)^{d(e_{p_1},j)} \frac{|G \setminus e_{p_1} (C)|}{n} + \sum_{t=2}^k (-1)^{d(e_{p_t},j)} \frac{|G \setminus e_{p_t} (C)|}{n}.\]
Since $d(e_{p_1},j)=1+d(i,j)$ and $d(e_{p_t},j)=d(i,j)$ for $t=2,3,\ldots,k$, the above becomes         
          \begin{align*}
          &  -(-1)^{d(i,j)} \frac{|G \setminus e_{p_1} (C)|}{n} + \sum_{t=2}^k (-1)^{d(i,j)} \frac{|G \setminus e_{p_t} (C)|}{n}\\
          &=  \frac{(-1)^{d(i,j)}}{n} \left( -|G \setminus e_{p_1} (C)| + \sum_{t=2}^k |G \setminus e_{p_t} (C)| \right)\\
          &=  \frac{(-1)^{d(i,j)}}{n} \left( -|G \setminus e_{p_1} (C)| + |G \setminus e_{p_1} (C)| - 1 \right)\\
          &=  \frac{-(-1)^{d(i,j)}}{n}.
        \end{align*}

Case 4. $i\in C$ and $j\notin C$\\
The proof in this case is similar to that of Case 3.
\end{enumerate}
\end{proof}

The preceding theorem gave $MH$. The following result gives a combinatorial formula for $HM$. In a connected graph, the distance between two edges $e_i$ and $e_j$, denoted by $d(e_i,e_j)$,  is the number of edges on a shortest path between a vertex on $e_i$ and a vertex on $e_j$.

\begin{theorem}\label{HM}
Let $G$ be an even unicyclic graph on $n$ vertices $1,2,\ldots,n$ with $n$ edges $e_1,e_2,\ldots,e_n$. Suppose $C$ is the cycle of $G$ and $M$ is the incidence matrix of $G$. For the matrix $H$ defined in (\ref{m^+ formula}), $HM$ is given by
\begin{equation}\label{HMeq}
(HM)_{i,j}= \frac{(-1)^{d(e_i,e_j)}}{|C|}\begin{cases} |C|   & \text{ if } e_i= e_j\notin C\\
|C|-1   & \text{ if } e_i =e_j\in C\\
1   & \text{ if } e_i \in C \text{ and } e_j\in C,\; i\neq j\\
0   & \text{ otherwise. } 
\end{cases}
\end{equation}
\end{theorem}

\begin{proof}
Let $e_i=\{r_i,s_i\}$ and $e_j=\{r_j,s_j\}$.\\

Case 1. $i=j$ \\
Note that $(HM)_{i,i}= h_{i,r_i} +h_{i,s_i}$.\\

Subcase (a) $e_i\notin C$\\
Since $d(e_i,r_i)=d(e_i,s_i)=0$ and $(r_i,s_i)$ is in $V(G\setminus e_i (C))\times V(G\setminus e_i [C])$ or in $V(G\setminus e_i [C])\times V(G\setminus e_i (C))$, we have

$$(HM)_{i,i}=\frac{1}{n}|G\setminus e_i (C)|+\frac{1}{n}|G\setminus e_i [C]|=\frac{1}{n}n=1.$$ 

Subcase (b) $e_i\in C$
\begin{align*}
    (HM)_{i,i} =& h_{i,r_i} + h_{i,s_i}\\
    =& \frac{1}{n|C|} \left[(-1)^{d_{G\setminus e_i}(r_i,r_i)} \left(-nd_{G\setminus e_i}(r_i,r_i^*) +\sum\limits_{t\in C} n_t d_{G\setminus e_i}(r_i,t)  \right) \right.\\
    &+ \left. (-1)^{d_{G\setminus e_i}(r_i,s_i)} \left(-nd_{G\setminus e_i}(r_i,s_i^*) +\sum\limits_{t\in C} n_t d_{G\setminus e_i}(r_i,t)  \right) \right].
    \end{align*}
Since $d_{G\setminus e_i}(r_i,r_i)=0$ and $d_{G \setminus e_i}(r_i,s_i)= |C|-1$ is odd,  the above becomes  
    \begin{align*}
    & \frac{1}{n|C|} \left[\sum\limits_{t\in C} n_t d_{G\setminus e_i}(r_i,t)  - \left(-n(|C|-1) +\sum\limits_{t\in C} n_t d_{G\setminus e_i}(r_i,t)  \right) \right]\\
    =& \frac{1}{n|C|} \left[n(|C|-1) + \sum\limits_{t\in C} n_t d_{G\setminus e_i}(r_i,t) - \sum\limits_{t\in C} n_t d_{G\setminus e_i}(r_i,t)  \right]\\
    =& \frac{1}{n|C|} \left[n(|C|-1)  \right]\\
    =& \frac{|C|-1}{|C|}. 
\end{align*}

Case 2. $i\neq j$\\

Subcase (a) $e_i\notin C$ and $e_j\notin C$\\
Note that $r_j$ and $s_j$ both are either in $G\setminus e_i (C)$ or in $G\setminus e_i [C]$ and $d(e_i,r_j)=d(e_i,s_j)\pm 1$. Then 
$$(HM)_{i,j}= h_{i,r_j} +h_{i,s_j}=\frac{(-1)^{d(e_i,r_j)} G\setminus e_i [C]}{n} +\frac{(-1)^{d(e_i,s_j)} G\setminus e_i [C]}{n}=0$$
or 
$$(HM)_{i,j}= h_{i,r_j} +h_{i,s_j}=\frac{(-1)^{d(e_i,r_j)} G\setminus e_i (C)}{n} +\frac{(-1)^{d(e_i,s_j)} G\setminus e_i (C)}{n}=0.$$

\medskip
Subcase (b) $e_i\in C$ and $e_j\in C$\\
Without loss of generality, let a shortest path between $e_i$  and $e_j$ be the shortest path between $r_i$ and $s_j$. Then $d(e_i,e_j)=d(r_i,s_j)=d_{G\setminus e_i}(r_i,s_j)$.

\begin{align*}
(HM)_{i,j}&= h_{i,r_j} +h_{i,s_j} \\ 
&= \frac{1}{n|C|}(-1)^{d_{G\setminus e_i}(r_i,r_j)} \left(-nd_{G\setminus e_i}(r_i,r_j^*) +\sum\limits_{t\in C} n_t d_{G\setminus e_j}(r_i,t)  \right)\\
&\;\;\;\; +\frac{1}{n|C|}(-1)^{d_{G\setminus e_i}(r_i,s_j)} \left(-nd_{G\setminus e_i}(r_i,s_j^*) +\sum\limits_{t\in C} n_t d_{G\setminus e_j}(r_i,t)  \right)
\end{align*}
Since $r_i^*=r_i$, $s_i^*=s_i$, and $d_{G\setminus e_i}(r_i,r_j)=d_{G\setminus e_i}(r_i,s_j)+1=d(e_i,e_j)+1$, the above becomes
\begin{align*}
&\;\;\;\;\frac{1}{n|C|}(-1)^{d(e_i,e_j)+1} \left(-n(d(e_i,e_j)+1) +\sum\limits_{t\in C} n_t d_{G\setminus e_j}(r_i,t)  \right)\\
&\;\;\;\; +\frac{1}{n|C|}(-1)^{d(e_i,e_j)} \left(-nd(e_i,e_j) +\sum\limits_{t\in C} n_t d_{G\setminus e_j}(r_i,t)  \right)\\
&= \frac{1}{n|C|}(-1)^{d(e_i,e_j)} \left(nd(e_i,e_j)+n -\cancel{\sum\limits_{t\in C} n_t d_{G\setminus e_j}(r_i,t)}
-nd(e_i,e_j) +\cancel{\sum\limits_{t\in C} n_t d_{G\setminus e_j}(r_i,t)} \right)\\
&= \frac{1}{n|C|}(-1)^{d(e_i,e_j)} n\\
&= \frac{1}{|C|}(-1)^{d(e_i,e_j)}.
\end{align*}

\medskip
Subcase (c) $e_i\notin C$ and $e_j\in C$\\
In this case, $r_j$ and $s_j$ are in $G\setminus e_i [C]$ and $d(e_i,r_j)=d(e_i,s_j)\pm 1$. Then 
$$(HM)_{i,j}= h_{i,r_j} +h_{i,s_j}=\frac{(-1)^{d(e_i,r_j)} G\setminus e_i (C)}{n} +\frac{(-1)^{d(e_i,s_j)} G\setminus e_i (C)}{n}=0.$$

Subcase (d) $e_i\in C$ and $e_j\notin C$\\
The proof is similar to that of Subcase (c).
\end{proof}

Now we are ready to state and prove our main result.

\begin{theorem}\label{M^+=H}
Let $G$ be an even unicyclic graph on $n$ vertices $1,2,\ldots,n$ and $n$ edges $e_1,e_2,\ldots,e_n$ with the cycle $C$ and the incidence matrix $M$. Then the matrix $H$ defined in (\ref{m^+ formula}) is the Moore-Penrose inverse of $M$. 
\end{theorem}
\begin{proof}
Since $G$ is an even unicyclic graph, $G$ is bipartite. Then by Theorem \ref{MH even unicyclic}, $$MH=I_n-\frac{1}{n}[(-1)^{d(i,j)}] \text{ and } H[(-1)^{d(i,j)}]=O.$$
Then $HMH=H$ and $(MH)^T=MH$. To prove $H=M^+$, it suffices to show that $HM$ is symmetric and $MHM=M$. Since  $[(-1)^{d(i,j)}]M=O$ by Observation \ref{Obs on M}, 
$$MHM=M-\frac{1}{n}[(-1)^{d(i,j)}]M=M.$$

It remains to show that $HM$ is symmetric which is evident from (\ref{HMeq}) in Theorem \ref{HM}. 
\end{proof}

By (\ref{m^+ formula}) and Theorem \ref{HM}, we have the following corollary:
\begin{corollary}
Let $G$ be an even unicyclic graph on $n$ vertices $1,2,\ldots,n$ and $n$ edges $e_1,e_2,\ldots,e_n$ with the cycle $C$. Suppose $M$ is the incidence matrix of $G$ with its Moore-Penrose inverse $M^+=[m_{ij}^+]$. Then the following hold:
\begin{enumerate}
    \item[(a)] $m_{ij}^+=\frac{n-1}{n}$ if and only if edge $e_i$ is a pendant edge incident with pendant vertex $j$.
    
    \item[(b)] The $(i,i)$-entry of $M^+M$ is $\frac{|C|-1}{|C|}$ if and only if edge $e_i$ is on $C$.
\end{enumerate}
\end{corollary}

\section{Open Problems}
We found a combinatorial formula for the Moore-Penrose inverse $M^+$ of the incidence matrix $M$ of an even unicyclic graph.  Using $M^+$, we can find the Moore-Penrose inverses of the signless Laplacian $Q=MM^+$ and signless edge-Laplacian $S=M^+M$ as follows:
\[Q^+=(MM^T)^+=(M^T)^+M^+=(M^+)^TM^+,\] \[S^+=(M^TM)^+=M^+(M^T)^+=M^+(M^+)^T.\]

But it still remains an open problem to find simple and compact combinatorial formulas for $Q^+$ and $S^+$ for even unicyclic graphs (like that in Theorem 3.5 and Theorem 3.9 in \cite{Hessert1}). It is just a small part of the bigger problem of finding the same for bipartite graphs.

Another open problem is to extend Bapat's work on trees \cite{B} to unicyclic graphs: Find combinatorial formulas for the Moore-Penrose inverse of an oriented incidence matrix $N$ and the Laplacian matrix $L=NN^T$ of a unicyclic graph.


\end{document}